\documentclass[review]{elsarticle}

\usepackage{lineno,hyperref}
\usepackage{amsmath,amssymb}
\usepackage{graphicx}
\usepackage{mathtools}
\usepackage{tikz}
\usepackage{subfigure}
\usepackage{mathptmx}
\usepackage{latexsym}
\usepackage{amsthm}
\newtheorem{theorem}{Theorem}
\newtheorem{assumption}[theorem]{Assumption}
\newtheorem{definition}[theorem]{Definition}
\newtheorem{remark}[theorem]{Remark}
\newtheorem{lemma}[theorem]{Lemma}

\modulolinenumbers[5]

\journal{Journal of \LaTeX\ Templates}

\begin{document}

\begin{frontmatter}
\title{Existence of weak solutions for a nonlocal pseudo-parabolic model for Brinkman two-phase flow in asymptotically flat porous media}


\author[mymainaddress]{Alaa Armiti-Juber \corref{mycorrespondingauthor}}
\cortext[mycorrespondingauthor]{Corresponding author}
\ead{alaa.armiti@mathematik.uni-stuttgart.de}

\author[mymainaddress]{Christian Rohde}

\address[mymainaddress]{Institute for Applied Analysis and Numerical Simulation, University of Stuttgart, Pfaffenwaldring 57, 70569 Stuttgart, Germany}

\begin{abstract}
We study a nonlocal evolution equation that involves a pseudo-parabolic third-order term. The equation models almost uni-directional two-phase flow in Brinkman regimes. We prove the existence of weak solutions for this equation. We also give a series of numerical examples that demonstrate the ability of the equation to support overshooting like in \cite{vanDuijn2013} and explore the behavior of solutions in various limit regimes.  
\end{abstract}


\end{frontmatter}


\section{Introduction}
\label{intro}
We consider the homogenized flow of two incompressible and immiscible phases in a rectangular porous media domain $\Omega_{\gamma}\in\mathbb{R}^2$ with $\gamma>0$ being the ratio of the vertical length to the horizontal length (see Figure \ref{fig:Omega}). According to e.g. \cite{Helmig1997} governing equations are given by
\begin{equation}\left.
\begin{array}{rl}
\partial_{t}S+\nabla\cdot\big(\textbf{V}f(S)\big)-\beta\nabla\cdot\big(H(S)\nabla S +\tau\beta \partial_t \nabla S\big)=&0,\\
\textbf{V}=&-\lambda_{tot}(S)\nabla p,\\
\nabla\cdot\textbf{V}=&0
\end{array}\right.
\label{eq:TP}
\end{equation}
in $\Omega_{\gamma}\times (0,T)$, where $T>0$ is the end time. The unknowns here are the saturation $S=S(x,z,t) \in [0,1]$ of the wetting phase and the global pressure $p=p(x,z,t)\in \mathbb{R}$. The total velocity $\textbf{V}=\textbf{V}(x,z,t)\in\mathbb{R}^{2}$, for any $(x,z,t)\in \Omega_{\gamma}\times(0,T)$ consists of a horizontal component $(U)$ and a vertical component $(W)$, i.e., $\textbf{V}=(U,W)^T$. By $f=f(S)\in[0,1]$ we denote the given fractional flow function. Also the diffusion function $H=H(S)\in [0,\infty]$ is given and defined as $H(S)\coloneqq f(S)\lambda_{nw}(S) p_c'(S)$, where $\lambda_{nw}=\lambda_{nw}(S)\in [0,\infty)$ is the mobility of the nonwetting phase and $p_c=p_c(S)\in [0,\infty)$ is the capillary pressure function. We refer to \cite{Helmig1997} for a general introduction to two-phase flow in porous media as well as possible choices for the capillary pressure function and the mobilities, where the mobilities together with the fluids' viscosities determine the fractional flow function. The total mobility function $\lambda_{tot}=\lambda_{tot}(S) \in (0,\infty)$ is the mobility sum of both phases. Here $\beta>0$ is a small parameter and $\tau>0$ determines the flow regime. The case $\tau=0$ results in the so-called Darcy regime, while $\tau>0$ is referred to as the Brinkman regime \cite{Helmig1997}.

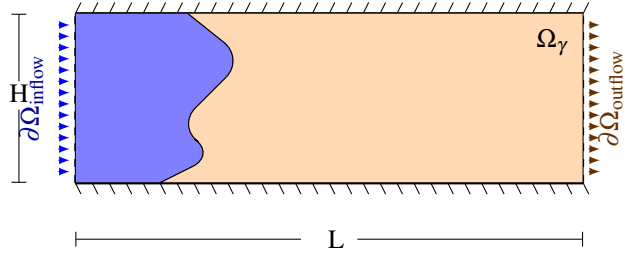
\begin{figure}
\centering
\begin{tikzpicture}[scale = 0.75,>=latex]
\draw[dashed, thick] (0,0) -- coordinate (y axis mid) (0,3);
\node[rotate=90, above=0.3cm, black!40!blue] at (y axis mid) {$\partial\Omega_{\text{inflow}}$};
\node[rotate=90, above=-7.4cm, black!60!orange] at (y axis mid) {$\partial\Omega_{\text{outflow}}$};
\draw[dashed, thick] (9,0) -- coordinate (x axis mid) (9,3);
\draw[very thick] (0,0) -- (9,0);
\draw[very thick] (0,3) -- (9,3); 

\draw[|-] (0,-1)--(4,-1) ;\draw[-|] (5,-1)--(9,-1);
\draw[very thick] (4.3,-1) node[anchor=west] {L};
\draw[|-] (-1,0)--(-1,1.3) ;\draw[-|] (-1,1.8)--(-1,3);
\draw[very thick] (-1,1.25) node[anchor=south] {H};

\draw[->,blue](-0.3,0.2)-- (-0.1,0.2);\draw[->,blue](-0.3,0.4)-- (-0.1,0.4);\draw[->,blue](-0.3,0.6)-- (-0.1,0.6);\draw[->,blue](-0.3,0.8)-- (-0.1,0.8);\draw[->,blue](-0.3,1)-- (-0.1,1);\draw[->,blue](-0.3,1.2)-- (-0.1,1.2);\draw[->,blue](-0.3,1.4)-- (-0.1,1.4);\draw[->,blue](-0.3,1.6)-- (-0.1,1.6);\draw[->,blue](-0.3,1.8)-- (-0.1,1.8);\draw[->,blue](-0.3,2)-- (-0.1,2);\draw[->,blue](-0.3,2.2)-- (-0.1,2.2);\draw[->,blue](-0.3,2.4)-- (-0.1,2.4);\draw[->,blue](-0.3,2.6)-- (-0.1,2.6);\draw[->,blue](-0.3,2.8)-- (-0.1,2.8);

\draw[->,black!60!orange](9.1,0.2)-- (9.3,0.2);\draw[->,black!60!orange](9.1,0.4)-- (9.3,0.4);\draw[->,black!60!orange](9.1,0.6)-- (9.3,0.6);\draw[->,black!60!orange](9.1,0.8)-- (9.3,0.8);\draw[->,black!60!orange](9.1,1)-- (9.3,1);\draw[->,black!60!orange](9.1,1.2)-- (9.3,1.2);\draw[->,black!60!orange](9.1,1.4)-- (9.3,1.4);\draw[->,black!60!orange](9.1,1.6)-- (9.3,1.6);\draw[->,black!60!orange](9.1,1.8)-- (9.3,1.8);\draw[->,black!60!orange](9.1,2)-- (9.3,2);\draw[->,black!60!orange](9.1,2.2)-- (9.3,2.2);\draw[->,black!60!black!60!orange](9.1,2.4)-- (9.3,2.4);\draw[->,black!60!orange](9.1,2.6)-- (9.3,2.6);\draw[->,black!60!orange](9.1,2.8)-- (9.3,2.8);

\draw[-] (0,0) -- (0.1,-0.2);\draw[-] (0.3,0) -- (0.4,-0.2);\draw[-] (0.6,0) -- (0.7,-0.2);\draw[-] (0.9,0) -- (1,-0.2);\draw[-] (1.2,0) -- (1.3,-0.2);\draw[-] (1.5,0) -- (1.6,-0.2);\draw[-] (1.8,0) -- (1.9,-0.2);\draw[-] (2.1,0) -- (2.2,-0.2);\draw[-] (2.4,0) -- (2.5,-0.2);\draw[-] (2.7,0) -- (2.8,-0.2);\draw[-] (3,0) -- (3.1,-0.2);\draw[-] (3.3,0) -- (3.4,-0.2);\draw[-] (3.6,0) -- (3.7,-0.2);\draw[-] (3.9,0) -- (4,-0.2);\draw[-] (4.2,0) -- (4.3,-0.2);\draw[-] (4.5,0) -- (4.6,-0.2);\draw[-] (4.8,0) -- (4.9,-0.2);\draw[-] (5.1,0) -- (5.2,-0.2);\draw[-] (5.4,0) -- (5.5,-0.2);\draw[-] (5.7,0) -- (5.8,-0.2);\draw[-] (6,0) -- (6.1,-0.2);\draw[-] (6.3,0) -- (6.4,-0.2);\draw[-] (6.6,0) -- (6.7,-0.2);\draw[-] (6.9,0) -- (7,-0.2);\draw[-] (7.2,0) -- (7.3,-0.2);\draw[-] (7.5,0) -- (7.6,-0.2);\draw[-] (7.8,0) -- (7.9,-0.2);\draw[-] (8.1,0) -- (8.2,-0.2);\draw[-] (8.4,0) -- (8.5,-0.2);\draw[-] (8.7,0) -- (8.8,-0.2);\draw[-] (9,0) -- (9.1,-0.2);

\draw[-] (0,3) -- (0.1,3.2);\draw[-] (0.3,3) -- (0.4,3.2);\draw[-] (0.6,3) -- (0.7,3.2);\draw[-] (0.9,3) -- (1,3.2);\draw[-] (1.2,3) -- (1.3,3.2);\draw[-] (1.5,3) -- (1.6,3.2);\draw[-] (1.8,3) -- (1.9,3.2);\draw[-] (2.1,3) -- (2.2,3.2);\draw[-] (2.4,3) -- (2.5,3.2);\draw[-] (2.7,3) -- (2.8,3.2);\draw[-] (3,3) -- (3.1,3.2);\draw[-] (3.3,3) -- (3.4,3.2);\draw[-] (3.6,3) -- (3.7,3.2);\draw[-] (3.9,3) -- (4,3.2);\draw[-] (4.2,3) -- (4.3,3.2);\draw[-] (4.5,3) -- (4.6,3.2);\draw[-] (4.8,3) -- (4.9,3.2);\draw[-] (5.1,3) -- (5.2,3.2);\draw[-] (5.4,3) -- (5.5,3.2);\draw[-] (5.7,3) -- (5.8,3.2);\draw[-] (6,3) -- (6.1,3.2);\draw[-] (6.3,3) -- (6.4,3.2);\draw[-] (6.6,3) -- (6.7,3.2);\draw[-] (6.9,3) -- (7,3.2);\draw[-] (7.2,3) -- (7.3,3.2);\draw[-] (7.5,3) -- (7.6,3.2);\draw[-] (7.8,3) -- (7.9,3.2);\draw[-] (8.1,3) -- (8.2,3.2);\draw[-] (8.4,3) -- (8.5,3.2);\draw[-] (8.7,3) -- (8.8,3.2);\draw[-] (9,3) -- (9.1,3.2);

\draw[fill=blue!50] (0,3)-- (0,0) -- (1.5,0)[rounded corners=10pt]--(2.5,0.5)--(1.8,1)--(3,2.2)--(2,3);
\draw[fill=orange!30]  (9,3)--(9,0)-- (1.5,0)[rounded corners=10pt]--(2.5,0.5)--(1.8,1)--(3,2.2) -- (2,3);
\draw[thick] (8,2.5)  node[anchor=west] {$\Omega_{\gamma}$}; 
\end{tikzpicture}
\caption{An illustration of the infiltration of the wetting phase into the domain $\Omega_{\gamma}$. For an asymptotically flat domain, the ratio $\gamma\coloneqq H/L$ tends to zero.}
\label{fig:Omega}
\end{figure}

In the case $\tau>0$, initial value problems for \eqref{eq:TP} have been already analyzed in \cite{Coclite2014}. In this paper, we focus on asymptotically flat domains. By such domains we mean to consider the limit $\gamma \rightarrow 0$ in \eqref{eq:TP}, see Figure \ref{fig:Omega}. The asymptotic limit has been formally addressed in \cite{Yortsos} for Darcy flow with $\tau=0$ and for Brinkman flow with $\tau>0$ in \cite{Armiti-Juber2018}. In the latter case the limit problem is given (after some space-time re-normalization and with a simplifying choice of a constant function $H(S)$) by the following nonlocal pseudo-parabolic equation
\begin{eqnarray}
\partial_{t}S +\partial_{x}\Bigl(f(S)U[S]\Bigr)+\partial_{z}\Bigl(f(S)W[S]\Bigr)-\beta \Delta S -\beta^2 \Delta \partial_{t}S=0\quad \text{ in }\Omega\times(0,T),
\label{eq:model}
\end{eqnarray}
where $\Omega=(0,1)\times(0,1)$ and we set $\tau=1$. Here, the saturation $S=S(x,z,t)\in [0,1]$ is the only unknown. The velocity components $U$ and $W$ are now nonlocal operators of saturation, given by
\begin{equation}
U[S]=\dfrac{\lambda_{tot}(S)}{\int_{0}^{1}\lambda_{tot}(S) dz},\quad \quad W[S]= -\partial_{x}\int_{0}^{z}U[S(\cdot,r,\cdot)]dr.
\label{eq:velocity}
\end{equation}
Note that the definition of the velocity components $U$ and $W$ in \eqref{eq:velocity} implies the incompressibility constraint
\begin{align}
 \partial_x U + \partial_z W=0.
 \label{eq:incompressible}
\end{align} 
\medskip

As mentioned above this nonlocal equation governs almost uni-directional two-phase flow in flat domains \cite{Armiti-Juber2018,Guo2014,Yortsos}. It is derived in \cite{Armiti-Juber2018} by applying asymptotic analysis, in terms of the height$-$length ratio of the domain to the Brinkman two-phase flow model \eqref{eq:TP}. This leads to a $z$-independent pressure function in the limit, a result that is usually called the vertical equilibrium assumption, see e.g. \cite{Guo2014}. This result is then used to reformulate the velocity components into nonlocal operators of saturation only, as in \eqref{eq:velocity}. In flat water aquifers, this assumption is called Dupuit-Forchheimer approximation. For example, it is utilized in \cite{vanDuijn2017} to derive a nonlocal differential equation that describes the movement of a sharp interface between fresh and salt groundwater.
\medskip

We call equation \eqref{eq:model} and \eqref{eq:velocity} as in \cite{Armiti-Juber2018}, the Brinkman Vertical Equilibrium model (BVE-model). It is shown there that this model is a proper reduction of the Brinkman two-phase model \eqref{eq:TP} in flat domains as it describes the vertical dynamics in the domain. In addition to this, it is computationally more efficient than the direct numerical simulation based on the full mixed hyperbolic-elliptic two-phase system for saturation and global pressure (see \cite{Armiti-Juber2018}). Note that the velocity in \eqref{eq:velocity} is computed from saturation directly, without solving an elliptic equation for the global pressure.
\medskip

Except of the nonlocal definition of the velocity components \eqref{eq:velocity}, model \eqref{eq:model} resembles the pseudo-parabolic model from \cite{HG1993}. This model supports the instability of overshooting of the invading wetting fronts (see e.g.~\cite{vanDuijn2013}). For the BVE-model overshooting is also observed \cite{Armiti-Juber2018}.

\medskip

Setting $\beta=0$, equation \eqref{eq:model} reduces to the nonlocal transport equation derived in \cite{Yortsos},
\begin{eqnarray}
\partial_{t}S +\partial_{x}\Bigl(f(S)U[S]\Bigr)+\partial_{z}\Bigl(f(S)W[S]\Bigr)=0\quad \quad\text{ in }\Omega\times(0,T),
\label{eq:YortsosModel}
\end{eqnarray}
where $U,\,W$ are still defined as in \eqref{eq:velocity}. This equation describes two-phase flow in flat domains of Darcy-type. We call this model the Darcy Vertical Equilibrium model (DVE-model). 
\medskip

To complete the BVE-model \eqref{eq:model}, \eqref{eq:velocity} we impose the initial and boundary conditions
\begin{align}
\begin{array}{r l l}
 S(\cdot,\cdot,0)&=S^0 &\text{ in } \Omega,\\
 S&=S_D & \text{ on }\partial \Omega\times[0,T]
\end{array}
 \label{eq:Bibc}
\end{align}
with $S^0=S^0(x,z)\in[0,1]$, $S_D=S_D(x,z,t)\in[0,1]$.
\medskip

In this paper, we are interested in proving the existence of weak solutions for the BVE-model \eqref{eq:model} and \eqref{eq:velocity} with the initial and boundary conditions \eqref{eq:Bibc}. For the pseudo-parabolic model from \cite{HG1993}, existence and uniqueness of weak solutions is proved in \cite{CaoPop2015,FanPop2010}. In fact, if the velocity component $W$ in \eqref{eq:velocity} would be Lipschitz continuous with respect to $S$, then the well-posedness of the BVE-model follows as in \cite{FanPop2010}. For the DVE-model, existence of weak solutions is still an open question due to the reduced regularity of the vertical velocity component $W$. This reduced regularity is a consequence of the differentiation operator $\partial_x$ in the definition of $W$ and the expected low regularity of a solution of a transport equation. Existence of weak solution for a regularization of the DVE-model, based on convoluting the velocity vector in \eqref{eq:velocity}, is investigated in \cite{Armiti-Juber2014}.
\medskip

The content of the paper is summarized as follows: in Section \ref{sec:assumptions} we give a list of assumptions on the BVE-model with the initial and boundary conditions \eqref{eq:Bibc}, propose a definition of weak solutions for the model and prove a few properties of the velocity components $U$ and $W$. Then, we prove in Section \ref{sec:existence} the existence of weak solutions for the model. Finally, we show in Section \ref{sec:numerical} through numerical examples the ability of the BVE-model to support overshooting fronts and investigate the behavior of the solutions in various limit regimes.

\section{Assumptions and Preliminaries}
\label{sec:assumptions}
We summarize all assumptions that are required throughout the paper. Furthermore, an appropriate notion of weak solution is presented and some preliminary for the velocity equations of weak solutions are provided. 
\begin{assumption}
\label{ass:Bexistence}
\begin{enumerate}
 \item The bounded domain $\Omega\subset\mathbb{R}^2$ has a Lipschitz continuous boundary $\partial \Omega$ and $0<T<\infty$.
 \item We require $S^0\in H^1_0(\Omega)$ and $S_D=0$.
 \item The fractional flow function $f\in C^1((0,1))$ is Lipschitz continuous, bounded, nonnegative and monotone increasing, such that there exist numbers $M,\,L>0$ with $f\leq M,\,f'\leq L$.
 \item The total mobility function $\lambda_{tot}\in C^1((0,1))$ is Lipschitz continuous, bounded and strictly positive, such that there exist numbers $a,\, M,\,L>0$ with $0<a<\lambda_{tot}\leq M$ and $\lvert\lambda_{tot}'\rvert\leq L$.
\end{enumerate} 
\end{assumption}
The BVE-model can be also extended to domains $\Omega\subset\mathbb{R}^3$, which leads to a third velocity component with a double integral. In the following, we denote $\Omega_T=\Omega\times(0,T)$.
 
\begin{definition}{(Weak Solution)}
A function $S\in H^1(0,T;H_0^{1}(\Omega))$ is called a weak solution of the BVE-model \eqref{eq:model}, \eqref{eq:velocity} and \eqref{eq:incompressible} with the initial and boundary conditions \eqref{eq:Bibc} if the following conditions hold,
\begin{enumerate}
 \item $U[S],\,W[S] \in L^2(\Omega_T)$ and 
 \begin{align}
  \int_{0}^{T}\int_{\Omega} \bigl(\partial_{t}S \phi- f(S) U[S]\partial_x \phi\bigr. & \bigr. - f(S)W[S]\partial_z\phi+ \beta \nabla S \cdot\nabla \phi \bigl) \,dx\,dz\,dt\nonumber\\& +\beta^2\int_{0}^{T}\int_{\Omega} \nabla \partial_{t}S \cdot\nabla\phi \,dx\,dz\,dt=0,
 \label{eq:def-cond1}
\end{align}
for all test functions $\phi\in L^{2}(0,T;H^{1}_0(\Omega))$.
\item The weak incompressibility property 
\begin{align}
       \int_{0}^{T}\int_{\Omega} \left( U[S]\partial_x \phi + W[S]\partial_z\phi\right) \,dx\,dz\,dt=0,
       \label{eq:def-cond2}
      \end{align}
holds for all test functions $\phi\in L^{2}(0,T;H^{1}_0(\Omega))$.
\item $S(\cdot,\cdot,0)=S^0$ almost everywhere in $\Omega$.
\end{enumerate}
\label{def:Bweaksolution}
\end{definition}

\begin{remark}
 Note that the integral $\int_0^1{\lambda_{tot}}\big(S(\cdot,z)\big)\,dz$ in the definition of the velocity components $U, ~W$ is an integral over a set of measure zero with respect to the $z$-coordinate. However, it is well-defined in the trace sense such that for all $x\in(0,1)$ there exists a bounded linear operator $T_x:H^1(\Omega) \rightarrow L^2\big((0,1)\big)$ and a constant $C>0$ satisfying
\begin{align*}
 \Vert T_x S\Vert_{L^2((0,1))}\leq C \Vert S\Vert_{H^1(\Omega)}.
\end{align*}
 \end{remark}
 
\begin{lemma}
For any weak solution of the BVE-model \eqref{eq:model}, \eqref{eq:velocity} and \eqref{eq:incompressible} the velocity components $U$ and $W$ in \eqref{eq:velocity} satisfy the following properties:
\begin{enumerate}
 \item $U$ is bounded with $ \Vert U[Q]\Vert_{L^{\infty}(\Omega_T)}\leq \frac{M}{a}$ for any function $Q\in L^2(\Omega_T)$.
 \item For any functions $Q_1,\,Q_2\in L^2(\Omega_T)$, the horizontal velocity $U$ satisfies
 \[\left\Vert U[Q_1]-U[Q_2]\right\Vert_{L^2(\Omega_T)}\leq \frac{2ML}{a^2}\Vert Q_1-Q_2\Vert_{ L^2(\Omega_T)}.\]
 \item For any function $Q\in L^2(0,T; H^1(\Omega))$, the vertical velocity $W$ satisfies the growth condition 
 \begin{align*}
 \Vert W[Q]\Vert_{L^2(\Omega_T))}\leq \frac{2ML}{a^2} \left\Vert\partial_x Q \right\Vert_{ L^2(\Omega_T)} .
\end{align*}
\end{enumerate}
\label{lem:B-more-assumptions}
\end{lemma}
\begin{proof}
 \begin{enumerate}
  \item Using Assumption \ref{ass:Bexistence}(4) we have
   \begin{align*}
     \Vert U[Q]\Vert_{L^{\infty}(\Omega_T)} =  \left\Vert \frac{\lambda_{tot}(Q)}{\int_0^1 \lambda_{tot}(Q(\cdot,z,\cdot))\,dz}\right\Vert_{L^{\infty}(\Omega_T)}\leq \frac{M}{a}.
   \end{align*}
   
  \item Using the triangle inequality and Assumption \ref{ass:Bexistence}(4), we have
\begin{align*}
 \Vert U[Q_1]-&\,U[Q_2]\Vert_{L^2(\Omega_T)}\\=&~\left\Vert \frac{\lambda_{tot}(Q_1)}{\int_0^1\lambda_{tot}(Q_1)\,dz}- \frac{\lambda_{tot}(Q_2)}{\int_0^1\lambda_{tot}(Q_2)\,dz}\right\Vert_{L^2(\Omega_T)},\\  \leq &~\left\Vert \frac{\lambda_{tot}(Q_1)\int_0^1 \bigl(\lambda_{tot}(Q_2)-\lambda_{tot}(Q_1) \bigr)dz}{\int_0^1\lambda_{tot}(Q_1)\,dz\int_0^1\lambda_{tot}(Q_2)\,dz}\right\Vert_{L^2(\Omega_T)}\\&+ \left\Vert \frac{\int_0^1\lambda_{tot}(Q_1)\,dz \bigl(\lambda_{tot}(Q_1)-\lambda_{tot}(Q_2) \bigr)}{\int_0^1\lambda_{tot}(Q_1)\,dz\int_0^1\lambda_{tot}(Q_2)\,dz}\right\Vert_{L^2(\Omega_T)},\\ \leq &~ \frac{M}{a^2}\left\Vert\int_0^1 \lambda'_{tot}(Q)\bigl(Q_2-Q_1 \bigr)dz\right\Vert_{L^2(\Omega_T)} + \frac{ML}{a^2}\left\Vert Q_2-Q_1 \right\Vert_{L^2(\Omega_T)}, 
\end{align*}
for some $Q\in L^2(\Omega_T)$. Note that the first term in the above inequality is constant in the vertical direction. Applying Jensen's inequality, then Fubini's inequality to this term yields
\begin{align*}
 \Vert U[Q_1]-U[Q_2]\Vert_{L^2(\Omega_T)}\leq \frac{2ML}{a^2}\left\Vert Q_2-Q_1 \right\Vert_{L^2(\Omega_T)}.
\end{align*}

\item Using the Lipschitz continuity of $\lambda_{tot}$ by Assumption \ref{ass:Bexistence}(4) we can apply the chain rule on $\partial_x \lambda_{tot}(Q)$. Then, using the triangle inequality, we have for any $z\in (0,1)$
\begin{align*}
 \Vert W[Q]\Vert_{L^2(\Omega_T))}=&~ \left\Vert -\partial_{x}\frac{\int_{0}^{z} \lambda_{tot}(Q(\cdot,r,\cdot))\,dr}{\int_0^1 \lambda_{tot}(Q(\cdot,r,\cdot))\,dr}\right\Vert_{L^2(\Omega_T)} ,\\ =& ~\left\Vert \frac{\int_{0}^{z} \lambda'_{tot}(Q(\cdot,r,\cdot))\partial_xQ(\cdot,r,\cdot)\,dr\int_0^1 \lambda_{tot}(Q(\cdot,r,\cdot))dr}{\left(\int_0^1 \lambda_{tot}(Q(\cdot,r,\cdot))\,dr\right)^2}\right\Vert_{L^2(\Omega_T)}\\&+  \left\Vert \frac{\int_{0}^{1} \lambda'_{tot}(Q(\cdot,r,\cdot))\partial_xQ(\cdot,r,\cdot)\,dr\int_0^z \lambda_{tot}(Q(\cdot,r,\cdot))dr}{\left(\int_0^1 \lambda_{tot}(Q(\cdot,r,\cdot))\,dr\right)^2}\right\Vert_{L^2(\Omega_T)},\\ \leq &~ \frac{ML}{a^2}\left(\left\Vert \int_0^z \partial_xQ(\cdot,r,\cdot)\,dr \right\Vert_{L^2(\Omega_T)}+ \left\Vert \int_0^1 \partial_xQ(\cdot,r,\cdot)\,dr \right\Vert_{L^2(\Omega_T)}\right).
\end{align*}
Applying Jensen's inequality and Fubini's inequality in any yields
\begin{align*}
  \Vert W[Q]\Vert_{L^2(\Omega_T))} \leq &\, \frac{2ML}{a^2}\int_0^1 \left\Vert \partial_xQ \right\Vert_{L^2(\Omega_T)}\,dr\leq \frac{2ML}{a^2} \left\Vert \partial_xQ \right\Vert_{L^2(\Omega_T)}.
\end{align*}
 \end{enumerate}
\end{proof}

\section{Existence of Weak Solutions}
\label{sec:existence}
In this section we prove the existence of weak solutions $S\in H^1(0,T;H^{1}(\Omega))$ for the BVE-model \eqref{eq:model}, \eqref{eq:velocity} and \eqref{eq:incompressible} with the initial and boundary conditions \eqref{eq:Bibc}. In Section \ref{sec:approximated}, we approximate the time derivatives in the model using backward differences, and apply Galerkin's method to the resulting series of elliptic problems. After that, we prove the existence of discrete solutions for the approximate problem. In Section \ref{sec:a priori}, we show that the sequence of discrete solutions fulfills a set of a priori estimates. These estimates are used in Section \ref{sec:convergence} to conclude the strong convergence of the sequence. Finally, we verify that the strong limit is a weak solution of the BVE-model.

\subsection{An Approximate BVE-Model}
\label{sec:approximated}
For $N\in\mathbb{N}$, $\Delta t\coloneqq T/N$, and any $t\in(0,T)$ we use the backward difference $\frac{S(t)-S(t-\Delta t)}{\Delta t}$ to approximate the time derivative $\partial_t S$. Then, equation \eqref{eq:model} can be approximated by 
\begin{align}
  \frac{S(t)-S(t-\Delta t)}{\Delta t} + \partial_x\Bigl( f(S(t)) U[S(t)]\Bigr) &+ \partial_z\Bigl(f(S(t))W[S(t)]\Bigr)-\beta\Delta S(t) \nonumber\\&-\beta^2\,  \frac{\Delta S(t)-\Delta S(t-\Delta t)}{\Delta t}=0
 \label{eq:Belliptic}
\end{align}
for $t\in (\Delta t, T+\Delta t)$.

Let $t$ be arbitrary but fixed. Then we consider weak solutions of equation \eqref{eq:Belliptic} from the Hilbert space $V(\Omega)\coloneqq H_{0}^{1}(\Omega)$. Let a countable orthonormal basis of $V$ be given by $\{w_{i}\}_{i\in\mathbb{N}}$. By applying Galerkin's method to \eqref{eq:Belliptic}, the solution space $V(\Omega)$ is projected onto a finite dimensional space $V_M(\Omega)$ spanned by the finite number of functions $w_i,~i=1,...,M$. For $\Delta t>0$ and a positive integer $M$, we search a function
        \begin{equation}
        S^{\Delta t}_{M}(x,z,t):=\sum_{i=1}^{M}c^{\Delta t}_{M,i}(t) w_{i}(x,z),
	\label{eq:Bdiscretesolution}
        \end{equation}
where the unknown coefficients $c^{\Delta t}_{M,i}\in L^{\infty}((0,T)),\, i=1,\dots,M$, are chosen such that for almost all $t\in(0,T)$ the relation 
	\begin{align}
	\int_{\Omega}& \left(S^{\Delta t}_M(t)-S_M^{\Delta t}(t-\Delta t)\right) w_i - \Delta t f(S^{\Delta t}_M(t)) \left(U[S^{\Delta t}_M(t)]\partial_x  w_i+ W[S^{\Delta t}_M(t)]\partial_z w_i\right)dx\,dz \nonumber\\& +\int_{\Omega}\left(\beta\Delta t \nabla S^{\Delta t}_M(t)+\beta^2 \nabla (S^{\Delta t}_M(t)-S_M^{\Delta t}(t-\Delta t))\right)\cdot\nabla w_i\,dx\,dz=0
	 \label{eq:Bapproximate1*}
	\end{align}
holds for all $i=1,...,M$, with 
\begin{align}
 U[S^{\Delta t}_M(t)(x,z)]=\dfrac{\lambda_{tot}\big(S^{\Delta t}_M(t)(x,z)\big)}{\int_{0}^{1}\lambda_{tot}\big(S^{\Delta t}_M(t)(x,r)\big) dr}, W[S^{\Delta t}_M(t)(x,z)]= -\partial_{x}\int_{0}^{z}U[S^{\Delta t}_M(t)(x,r)]dr,
\label{eq:Bapproximate2*}
\end{align}
for almost all $t\in(0,T)$ and $(x,z)\in\Omega$. The function $S^{\Delta t}_M$ is also required to satisfy the weak incompressibility relation 
\begin{align}
  \int_{\Omega} U[S^{\Delta t}_M]\partial_x w_i + W[S^{\Delta t}_M]\partial_z w_i\,dx\,dz&=0, \quad \text{ for all } i=1,...,M.
\label{eq:B-incompressible1}
\end{align}
Further more we define
        \begin{equation}
         S^{\Delta t}_M(t)=S_M^0,\quad \text{ for }t \in ( -\Delta t,0] ,
         \label{eq:B-IC}
        \end{equation}
where $S_M^0$ is the $L^2$-projection of the initial data $S^0$ to the finite dimensional space $V_M(\Omega)$.

To prove the existence of a weak solution of the discrete problem \eqref{eq:Bapproximate1*}, \eqref{eq:Bapproximate2*}, \eqref{eq:B-incompressible1} and \eqref{eq:B-IC}, we need the following technical lemma on the existence of zeros of a vector field \cite{Evans}.

\begin{lemma}
Let $r>0$ and $\textbf{v}:\mathbb{R}^n\rightarrow\mathbb{R}^n$ be a continuous vector field, which satisfies $\textbf{v}(\textbf{x})\cdot \textbf{x}\geq 0$ if $|\textbf{x}|=r$. Then, there exists a point $\textbf{x}\in B(0,r)$ such that $\textbf{v}(\textbf{x})=\textbf{0}$.
\label{lem:Bvectorfield}
\end{lemma}

\begin{lemma}
For any $M,\,N\in\mathbb{N}$ and for almost all $t\in (0,T)$, if $S^{\Delta t}_M(t-\Delta t) \in V_M(\Omega) $ is known, then the discrete problem \eqref{eq:Bapproximate1*}, \eqref{eq:Bapproximate2*}, \eqref{eq:B-incompressible1} and \eqref{eq:B-IC} has a solution $S^{\Delta t}_M(t)\in V_M(\Omega)$ that satisfies
\begin{align}
 &\int_{\Omega} \big(S^{\Delta t}_M(t)-S_M^{\Delta t}(t-\Delta t)\big) \phi \,dx\,dz - \Delta t\int_{\Omega} f(S^{\Delta t}_M) U[S^{\Delta t}_M]\partial_x \phi+ f(S^{\Delta t}_M)W[S^{\Delta t}_M]\partial_z\phi\,dx\,dz\nonumber\\&+\int_{\Omega}\left[\beta\Delta t \nabla S^{\Delta t}_M+\beta^2 \nabla \left(S^{\Delta t}_M(t)-S_M^{\Delta t}(t-\Delta t)\right)\right]\cdot\nabla\phi\,dx\,dz=0,
 \label{eq:Bapproximate}
\end{align}
for all $\phi \in V_M(\Omega)$.
\end{lemma}
\begin{proof}
Before starting with the proof, we notify that $S^{\Delta t}_M(t-\Delta t)$ for $t\in(0,\Delta t]$ is well-defined by the choice of the initial condition \eqref{eq:B-IC}. Now, we define the vector field $\textbf{K}:\mathbb{R}^{M}\rightarrow\mathbb{R}^{M}$, $\textbf{K}=(k_1,...,k_M)^T$, and $\textbf{c}_M^{\Delta t}(t)=\big(c_{M,1}^{\Delta t}(t),\cdots,c_{M,M}^{\Delta t}(t)\big)^T$ of the unknown coefficients in equation \eqref{eq:Bdiscretesolution} such that, for almost all $t\in(0,T)$,
	\begin{align}
	k_i(\textbf{c}_M^{\Delta t}(t)):=&\int_{\Omega} \big(S^{\Delta t}_M(t)-S_M^{\Delta t}(t-\Delta t)\big) w_i \,dx\,dz \nonumber\\ &- \Delta t\int_{\Omega} f(S^{\Delta t}_M) U[S^{\Delta t}_M]\partial_x  w_i+ f(S^{\Delta t}_M)W[S^{\Delta t}_M]\partial_z w_i\,dx\,dz \nonumber\\&+ \int_{\Omega}\left(\beta\Delta t \nabla S^{\Delta t}_M + \beta^2 \nabla (S^{\Delta t}_M(t)-S_M^{\Delta t}(t-\Delta t))\right)\cdot\nabla w_i\,dx\,dz,
	\label{eq:Bpremitive}
	\end{align}
for all $ i=1,...,M$. The vector field $\textbf{K}$ is continuous by Assumptions \ref{ass:Bexistence}(3) and \ref{ass:Bexistence}(4) Moreover, using \eqref{eq:Bdiscretesolution}, we have
	\begin{align}
	\textbf{K}(\textbf{c}_M^{\Delta t}(t))&\cdot\textbf{c}_M^{\Delta t}(t)\nonumber\\=&\int_{\Omega} \big(S^{\Delta t}_M(t)-S_M^{\Delta t}(t-\Delta t)\big) S^{\Delta t}_M(t) \,dx\,dz\nonumber\\& - \Delta t\int_{\Omega} f(S^{\Delta t}_M)\left( U[S^{\Delta t}_M]\partial_x  S^{\Delta t}_M+ W[S^{\Delta t}_M]\partial_z S^{\Delta t}_M\right)\,dx\,dz,\nonumber \\&+ \int_{\Omega}\left(\beta\Delta t\nabla S^{\Delta t}_M(t)+\beta^2 \nabla \big(S^{\Delta t}_M(t)-S^{\Delta t}_M(t-\Delta t)\big)\right)\cdot\nabla S^{\Delta t}_M(t)\,dx\,dz.
	\label{eq:Belliptic-existence}
	\end{align}
Let $F(S)\coloneqq\int_0^{S} f(q)dq$, then the second term on the right side of \eqref{eq:Belliptic-existence} satisfies
 \begin{align*}
  \int_{\Omega} f(S^{\Delta t}_M) \big(U[S^{\Delta t}_M]\partial_x  S^{\Delta t}_M+ W[S^{\Delta t}_M]\partial_z S^{\Delta t}_M\big)\,dx\,dz&=\int_{\Omega} f(S^{\Delta t}_M)\textbf{V}[S^{\Delta t}_M]\cdot \nabla S^{\Delta t}_M\, dx\,dz,\\&=\int_{\Omega} \textbf{V}(S^{\Delta t}_M)\cdot\nabla F(S^{\Delta t}_M)\,dx\,dz,
 \end{align*}
where $\textbf{V}[S^{\Delta t}_M]\coloneqq (U[S^{\Delta t}_M],W[S^{\Delta t}_M])^T$. Using the Assumption \ref{ass:Bexistence}(2) and the property $F(0)=0$, the weak incompressibility equation \eqref{eq:B-incompressible1} with $w_i\coloneqq F(S^{\Delta t}_M)\in L^2(0,T;H_0^1(\Omega))$ 
implies that
\begin{align}
  \int_{\Omega} \textbf{V}(S^{\Delta t}_M)\cdot\nabla F(S^{\Delta t}_M)\,dx\,dz=0.
  \label{eq:useofincompress}
 \end{align}	
Substituting equation \eqref{eq:useofincompress} into \eqref{eq:Belliptic-existence}, then applying Cauchy's inequality yields
	\begin{align*}
	\textbf{K}(\textbf{c}_M^{\Delta t}(t))\cdot\textbf{c}_M^{\Delta t}(t)\geq&~\frac{1}{2} \Vert S^{\Delta t}_M\Vert^2_{L^2(\Omega)}+\left(\frac{\beta^2}{2}+\Delta t\beta \right)\Vert \nabla S^{\Delta t}_M\Vert^2_{L^2(\Omega)} -\frac{1}{2}\Vert S_M^{\Delta t}(t-\Delta t)\Vert^2_{L^2(\Omega)}\\&-\frac{\beta^2}{2}\Vert \nabla S_M^{\Delta t}(t-\Delta t)\Vert^2_{L^2(\Omega)} .
	\end{align*}
Equation \eqref{eq:Bdiscretesolution} and the orthonormality of $w_i,\,i\in \{1,\cdots,M\}$, yield
	\begin{align*}
	\textbf{K}\big(\textbf{c}_M^{\Delta t}(t)\big)\cdot\textbf{c}_M^{\Delta t}(t)\geq&~ \left(\frac{1}{2}+\frac{\beta^2}{2}+\Delta t\beta \right)|\textbf{c}_M^{\Delta t}|^2 -\frac{1}{2}\Vert S_M^{\Delta t}(t-\Delta t)\Vert_{L^2(\Omega)}\\&-\frac{\beta^2}{2}\Vert \nabla S_M^{\Delta t}(t-\Delta t)\Vert_{L^2(\Omega)} .
	\end{align*}
Note that $S_M^{\Delta t}(t-\Delta t)\in V_M(\Omega)$ is now given. Setting $r=|\textbf{c}_M^{\Delta t}(t)|$, we conclude that $\textbf{K}(\textbf{c}_M^{\Delta t}(t))\cdot\textbf{c}_M^{\Delta t}(t)\geq 0$ provided that $r$ is large enough. Thus, Lemma \ref{lem:Bvectorfield} ensures the existence of a vector $\textbf{c}_M^{\Delta t}(t) \in \mathbb{R}^M$ with $\textbf{K}(\textbf{c}_M^{\Delta t}(t))=\textbf{0}$. Using equation \eqref{eq:Bpremitive} we get the existence of an $S^{\Delta t}_M(t)$, that satisfies the discrete problem \eqref{eq:Bapproximate1*}, \eqref{eq:Bapproximate2*}, \eqref{eq:B-incompressible1} and \eqref{eq:B-IC}. 
\end{proof}

\subsection{A priori Estimates}
\label{sec:a priori}
So far, we proved the existence of a sequence $\{S^{\Delta t}_M\}_{M\in\mathbb{N},\,\Delta t>0} \subset V_M(\Omega)$ of solutions for the discrete problem \eqref{eq:Bapproximate1*}, \eqref{eq:Bapproximate2*}, \eqref{eq:B-incompressible1} and \eqref{eq:B-IC}. In the following, we prove some a priori estimates on the sequence that are essential for the convergence analysis in the next subsection. 

\begin{lemma}
If Assumption \ref{ass:Bexistence} holds, then the sequence of solutions $\{S^{\Delta t}_M\}_{M\in\mathbb{N},\,\Delta t>0}$ for the discrete problem \eqref{eq:Bapproximate1*}, \eqref{eq:Bapproximate2*}, \eqref{eq:B-incompressible1} and \eqref{eq:B-IC} satisfies
 \begin{align*}
\underset{t\in[0,T]}{\emph{ess}\sup}\left(\rVert S^{\Delta t}_M (t) \lVert^2_{L^2(\Omega)}+ \beta^2 \rVert \nabla S^{\Delta t}_M (t) \lVert^2_{L^2(\Omega)}\right)&+\beta\rVert \nabla S^{\Delta t}_M \lVert^2_{L^2(\Omega_T)}\\&\leq  \Vert S_M^0\Vert^2_{L^2(\Omega)}+\beta^2 \Vert \nabla S_M^0\Vert^2_{L^2(\Omega)}
 \end{align*} for all $M\in\mathbb{N}$ and $\Delta t>0$.
\label{lem:Bapriori1}
\end{lemma}
\begin{proof}
Multiplying equation \eqref{eq:Bapproximate1*} by $c_{M,i}^{\Delta t}$, summing for $i=1,...,M$, then integrating from $0$ to an arbitrary $\tau\in(0,T)$ yields
\begin{align}
&\dfrac{1}{\Delta t} \int_{0}^{\tau} \int_{\Omega} \left(S_M^{\Delta t}(t)-S_M^{\Delta t}(t-\Delta t) \right) S^{\Delta t}_M(t)\,dx\,dz\,dt- \int_0^{\tau}\int_{\Omega} \textbf{V}[S^{\Delta t}_M]f(S^{\Delta t}_M)\cdot \nabla S^{\Delta t}_M \,dx\,dz\,dt \nonumber \\ &+\int_{0}^{\tau} \int_{\Omega}\beta|\nabla S^{\Delta t}_M|^2+ \dfrac{\beta^2}{\Delta t} \left(\nabla S_M^{\Delta t}(t)-\nabla S_M^{\Delta t}(t-\Delta t) \right)\cdot\nabla S^{\Delta t}_M(t)\,dx\,dz\,dt=0.
\label{eq:B1-apriori1}
\end{align}
Using summation by parts, the first term on the left side of equation \eqref{eq:B1-apriori1} satisfies 
\begin{align*}
 \frac{1}{\Delta t}\int_0^{\tau}\int_{\Omega} \left(S_M^{\Delta t}(t)-S_M^{\Delta t}(t-\Delta t) \right) S^{\Delta t}_M(t)\,dx\,dz\,dt=\,& \frac{1}{2\Delta t}\int_{{\tau}-\Delta t}^{\tau}\int_{\Omega} (S^{\Delta t}_M(t))^2\,dx\,dz\,dt\\ &- \frac{1}{2\Delta t}\int_{-\Delta t}^0 \int_{\Omega} (S^{\Delta t}_M(t))^2\,dx\,dz\,dt.
\end{align*}
Since $S^{\Delta t}_M$ is a step function in time, the above equation simplifies to
\begin{align}
\dfrac{1}{\Delta t} \int_{0}^{\tau} \int_{\Omega} \left(S_M^{\Delta t}(t)-S_M^{\Delta t}(t-\Delta t) \right)& S^{\Delta t}_M(t)\,dx\,dz\,dt\nonumber\\&=\frac{1}{2}\int_{\Omega}\big(S^{\Delta t}_M(\tau)\big)^2-(S_M^0)^2\,dx\,dz.
\label{eq:B1-apriori2}
\end{align}
Similarly, we have
\begin{align}
 \int_0^{\tau}\int_{\Omega}  \frac{\nabla S_M^{\Delta t}(t)-\nabla S_M^{\Delta t}(t-\Delta t)}{\Delta t}\cdot\,& \nabla S^{\Delta t}_M(t)\,dx\,dz\,dt\nonumber\\&=\frac{1}{2}\int_{\Omega}|\nabla S^{\Delta t}_M(\tau)|^2-|\nabla S_M^0|^2\,dx\,dz.
 \label{eq:B1-apriori3}
\end{align}
Using the primitive $F(S)= \int_0^{S} f(q)\,dq$ and the weak incompressibility of the velocity \eqref{eq:B-incompressible1}, we obtain as in equation \eqref{eq:useofincompress} the relation
 \begin{align}
 \int_0^{\tau}\int_{\Omega} \textbf{V}[S^{\Delta t}_M]f(S^{\Delta t}_M)\cdot \nabla S^{\Delta t}_M \,dx\,dz\,dt=\int_0^{\tau}\int_{\Omega} \textbf{V}[S^{\Delta t}_M]\cdot \nabla F(S^{\Delta t}_M)\,dx\,dz\,dt=0.
 \label{eq:B1-apriori4}
 \end{align}
Since the time $\tau\in(0,T)$ is arbitrarily chosen, substituting equation \eqref{eq:B1-apriori2}, \eqref{eq:B1-apriori3}, and \eqref{eq:B1-apriori4} into \eqref{eq:B1-apriori1} yields
\begin{align*}
 \underset{\tau\in[0,T]}{\text{ess} \sup}~ \Big(\Vert S^{\Delta t}_M (\tau)\Vert_{L^2(\Omega)}+ \beta^2 \Vert \nabla S^{\Delta t}_M (\tau)\Vert_{L^2(\Omega)}\Big)&+ 2\beta\int_0^T\Vert \nabla S^{\Delta t}_M \Vert_{L^2(\Omega)}\,dt \\ & \leq \Vert S_M^0\Vert^2_{L^2(\Omega)}+\beta^2 \Vert \nabla S_M^0\Vert^2_{L^2(\Omega)},
\end{align*}
 for all $M\in\mathbb{N}$ and any $\Delta t >0$.
 \end{proof} 
 
In the following lemma, we prove an estimate on the approximate time derivatives $\tfrac{S^{\Delta t}_M(t) - S_M^{\Delta t}(t-\Delta t)}{\Delta t}$ and $\tfrac{\nabla S^{\Delta t}_M(t) - \nabla S_M^{\Delta t}(t-\Delta t)}{\Delta t}$, which depend on the parameter $\beta >0$.
\begin{lemma}
If Assumption \ref{ass:Bexistence} holds, then there exists a constant $C > 0$ independent of $N,\,m \in\mathbb{N}$ such that we have for almost all $t\in(0,T)$ for all $N,\,m \in\mathbb{N}$ the estimate
\begin{align*}
   \left\Vert S^{\Delta t}_M(t) - S_M^{\Delta t}(t-\Delta t)\right\Vert^2_{L^2(\Omega_T)}+ \beta^2 \left\Vert \nabla \left(S^{\Delta t}_M(t) - S_M^{\Delta t}(t-\Delta t)\right) \right\Vert^2_{L^2(\Omega_T)}\leq \frac{C}{\beta^2} \Delta t^2.
\end{align*}
\label{lem:Bapriori2}
\end{lemma}
\begin{proof}
Multiplying equation \eqref{eq:Bapproximate1*} by $\big(c_{M,i}^{\Delta t}(t)-c_{M,i}^{\Delta t}(t-\Delta t)\big)$, summing for $i=1,...,M$, then integrating from $0$ to $T$ yields
\begin{align*}
\lVert S^{\Delta t}_M(t)-S_M^{\Delta t}(t-\Delta t)\rVert^2_{L^2(\Omega_T)}&+\beta^2 \lVert\nabla \big(S^{\Delta t}_M(t)-S_M^{\Delta t}(t-\Delta t)\big)\rVert^2_{L^2(\Omega_T)}\\ =&~\,\Delta t\int_0^T\int_\Omega  \textbf{V}[S^{\Delta t}_M]f(S^{\Delta t}_M)\cdot\nabla\big( S^{\Delta t}_M(t)-S_M^{\Delta t}(t-\Delta t)\big)\,dx\,dz\,dt\\&-\beta\Delta t\int_0^T\int_{\Omega}  \nabla S^{\Delta t}_M \cdot\nabla \big( S^{\Delta t}_M(t)-S_M^{\Delta t}(t)\big) \,dx\,dz\,dt.
\end{align*}
Applying Cauchy's inequality to the right side of the equation above yields
\begin{align*}
\lVert S^{\Delta t}_M(t)-S_M^{\Delta t}(t-\Delta t)&\rVert^2_{L^2(\Omega_T)}+\beta^2 \lVert\nabla \big(S^{\Delta t}_M(t)-S_M^{\Delta t}(t-\Delta t)\big)\rVert^2_{L^2(\Omega_T)}\\  \leq&~ \dfrac{\Delta t^2 }{\beta^2}\Vert\textbf{V}[S^{\Delta t}_M]f(S^{\Delta t}_M)\Vert^2_{L^2(\Omega_T)}+ \dfrac{\beta^2}{4} \lVert\nabla \big(S^{\Delta t}_M(t)-S_M^{\Delta t}(t-\Delta t)\big)\rVert^2_{L^2(\Omega_T)} \\&+4\Delta t^2 \rVert \nabla S^{\Delta t}_M\lVert_{L^2(\Omega_T)} + \dfrac{\beta^2}{4} \lVert\nabla \big(S^{\Delta t}_M(t)-S_M^{\Delta t}(t-\Delta t)\big)\rVert^2_{L^2(\Omega_T)} .
\end{align*}
The growth conditions on $\textbf V=(U,W)^T$ from Lemma \ref{lem:B-more-assumptions} and the a priori estimate from Lemma \ref{lem:Bapriori1} give
\begin{align*}
\left\Vert S^{\Delta t}_M(t)\right. -& \left.S_M^{\Delta t}(t-\Delta t)\right\Vert^2_{L^2(\Omega_T)}+\dfrac{\beta^2}{2} \lVert\nabla \big(S^{\Delta t}_M(t)-S_M^{\Delta t}(t-\Delta t)\big)\rVert^2_{L^2(\Omega_T)} \\ & \leq \frac{\Delta t^2}{\beta^2}\Big( \frac{M^4|\Omega|T}{a^2}\Big)+ \frac{\Delta t^2}{\beta^2} \Big(\frac{4M^4L^2}{a^2} \Big)\Vert \nabla S^{\Delta t}_M \Vert^2_{L^2(\Omega_T)}+4\Delta t^2\Vert \nabla S^{\Delta t}_M \Vert^2_{L^2(\Omega_T)},\\& \leq \frac{\Delta t^2}{\beta^2}\Big( \frac{M^4|\Omega|T}{a^2}+ \big(\frac{4M^4L^2}{a^2} +4\beta^2\big)\big(\Vert  S_M^0 \Vert^2_{L^2(\Omega)}+ \beta^2\nabla \Vert S_M^0 \Vert^2_{L^2(\Omega)}\big)\Big),\\ & \leq C\,\frac{\Delta t^2}{\beta^2},
\end{align*}
where $C= \frac{M^4|\Omega|T}{a^2}+ \big(\frac{4M^4L^2}{a^2} +4\beta^2\big)\big(\Vert  S_M^0 \Vert^2_{L^2(\Omega)}+ \beta^2\nabla \Vert S_M^0 \Vert^2_{L^2(\Omega)}\big)$ for all $M\in\mathbb{N}$ and any $\Delta t >0$. 
\end{proof}

\subsection{Convergence Analysis}
\label{sec:convergence}
In this section, we show the convergence of the sequence $\{S^{\Delta t}_M\}_{M\in\mathbb{N},\,\Delta t>0}$, then prove that the limit is a weak solution of the BVE-model \eqref{eq:model}, \eqref{eq:velocity} and \eqref{eq:incompressible} with the initial and boundary conditions \eqref{eq:Bibc}.

\begin{theorem}
Let Assumption \ref{ass:Bexistence} be satisfied. Then, there exists a weak solution $S\in H^1(0,T;H^{1}(\Omega))$ of the initial boundary value problem \eqref{eq:model}, \eqref{eq:velocity}, \eqref{eq:incompressible} and \eqref{eq:Bibc} satisfying Definition \ref{def:Bweaksolution}.
 \label{thm:B-maintheorem2}
\end{theorem}

\begin{proof}
The uniform estimates in Lemma \ref{lem:Bapriori1} imply the existence of a weakly convergent subsequence of $\{ S^{\Delta t}_M\}_{M\in\mathbb{N},\,\Delta t>0}$, denoted in the same way, and a function $S\in L^2(0,T; H_0^1(\Omega))$ such that
\begin{align}
 S^{\Delta t}_M \rightharpoonup S \in L^2(0,T; H_0^1(\Omega)),
\end{align}
as $M\rightarrow \infty$ and $\Delta t\rightarrow 0$. In addition, Lemma \ref{lem:Bapriori2} implies $\partial_t S\in L^2(0,T; H_0^1(\Omega))$. Thus, we have the weak convergence result
  \begin{align}
 S^{\Delta t}_M \rightharpoonup S \in H^1(0,T; H_0^1(\Omega)).
 \label{eq:Bweakconv}
\end{align}
The Rellich-Kondrachov compactness theorem implies $H^1(0,T; H_0^1(\Omega))\Subset L^6(\Omega_T)$ and consequently $H^1(0,T; H_0^1(\Omega)) \Subset L^2(\Omega_T)$ due to the boundedness of the domain $\Omega_T$. Thus, from the weak convergence result \eqref{eq:Bweakconv}, we extract the strong convergence
\begin{align}
 S^{\Delta t}_M \rightarrow S \in L^2(\Omega_T).
 \label{eq:Bstrongconv}
\end{align}
This strong convergence and the a priori estimate from Lemma \ref{lem:Bapriori1} imply that the limit $S$ satisfies
\begin{align}
 S,\,\nabla S \in L^{\infty}(0,T;H_0^1(\Omega)).
\label{eq:B-limit}
\end{align}
Moreover, we have
\begin{align}
 S\in C([0,T];H^1_0(\Omega)).
 \label{eq:B-S-timecontinuous}
\end{align}

The next step in the proof is to show that the function $S\in H^1(0,T; H_0^1(\Omega))$ with $ S\in C([0,T];L^2(\Omega))$ fulfills the conditions in Definition \ref{def:Bweaksolution}. Thus, we consider an arbitrary test function $\phi\in L^2(0,T;V_m(\Omega))$ such that for a fixed integer $m$ and for almost all $t\in(0,T)$
\begin{align}
 \phi(t)=\sum_{i=1}^m c_i(t) w_i,
\end{align}
where $c_i\in L^{\infty}(0,T)$, $i=1,\cdots,m$, are given functions and $w_i\in H_0^1(\Omega)$, $i=1,\cdots,m$, belong to the orthonormal basis of the subspace $V_m(\Omega)$. Choosing $m<M$, multiplying equation \eqref{eq:Bapproximate1*} by $c_i(t)$, summing for $i=1,\cdots, m$, and then integrating with respect to time yields
\begin{align}
\dfrac{1}{\Delta t} &\int_{0}^{T} \int_{\Omega} \left(S_M^{\Delta t}(t)-S_M^{\Delta t}(t-\Delta t) \right)\phi\,dx\,dz\,dt\nonumber\\&- \int_0^{T}\int_{\Omega} U[S^{\Delta t}_M]f(S^{\Delta t}_M)\partial_x \phi \,dx\,dz\,dt - \int_0^{T}\int_{\Omega} W[S^{\Delta t}_M]f(S^{\Delta t}_M)\partial_z \phi \,dx\,dz\,dt\nonumber \\ &+\int_{0}^{T} \int_{\Omega}\nabla S^{\Delta t}_M\cdot\nabla \phi + \dfrac{\beta}{\Delta t} \nabla\left(S_M^{\Delta t}(t)-S_M^{\Delta t}(t-\Delta t) \right)\cdot \nabla \phi\,dx\,dz\,dt\nonumber\\=&~0.
\label{eq:B-weakdiscreate}
\end{align}
The strong convergence \eqref{eq:Bstrongconv} and the Lipschitz continuity of $f$ and $\lambda_{tot}$ imply
\begin{align}
\begin{array}{rll}
 \Vert f(S^{\Delta t}_M)- f(S)\Vert_{L^2(\Omega_T)}&\leq L\, \Vert S^{\Delta t}_M- S\Vert_{L^2(\Omega_T)} &\rightarrow 0,\vspace{5pt}\\
  \Vert \lambda_{tot} (S^{\Delta t}_M)- \lambda_{tot}(S)\Vert_{L^2(\Omega_T)} & \leq L\, \Vert S^{\Delta t}_M- S\Vert_{L^2(\Omega_T)}& \rightarrow 0.
\end{array}
\label{eq:B-strong-nonlinear}
\end{align}
Jensen's inequality and Fubini's theorem imply 
\begin{align*}
\int_0^T \int_0^1 \left|\int_{0}^{1} \lambda_{tot}\right.&\left.\left(S^{\Delta t}_M(t)(x,z)\right)dz-\int_{0}^{1}\lambda_{tot}\bigl(S(t)(x,z)\bigr)\,dz \right|^2 dx \,dt \\ \leq &  \int_0^T \int_0^1 \int_{0}^{1}\left| \lambda_{tot}\left(S^{\Delta t}_M(x,z,t)\right)-\lambda_{tot}\bigl(S(x,z,t)\bigr)\right|^2 dz\, dx\,dt,\\=&  \int_0^T\int_{\Omega}\left| \lambda_{tot}\left(S^{\Delta t}_M(x,z,t)\right)-\lambda_{tot}\bigl(S(x,z,t)\bigr)\right|^2 \,dx\,dz\,dt.
\end{align*}
Thus, the strong convergence of $\lambda_{tot}(S^{\Delta t}_M)$ in \eqref{eq:B-strong-nonlinear} implies that 
\begin{align}
 \int_{0}^{1} \lambda_{tot}\left(S^{\Delta t}_M(.,z,.)\right)\,dz\rightarrow \int_{0}^{1} \lambda_{tot}\bigl(S(\cdot,z,\cdot)\bigr)\,dz \quad \text{ in } L^2\big((0,1)\times(0,T)\big).
\end{align}
As the sequence $\int_{0}^{1} \lambda_{tot}\left(S^{\Delta t}_M(.,z,.)\right)\,dz$ is constant in the $z$-direction, we have 
\begin{align}
 \int_{0}^{1} \lambda_{tot}\left(S^{\Delta t}_M(.)(.,z)\right)\,dz\rightarrow \int_{0}^{1} \lambda_{tot}\bigl(S(.,z,.)\bigr)\,dz \quad \text{ in } L^2(\Omega_T).
 \label{eq:B-strong-nonlinear1}
\end{align}
To prove the strong convergence $U[S^{\Delta t}_M]\rightarrow U[S]$ in $L^2(\Omega_T)$, we use the notation $A[R](x,t)\coloneqq\int_0^1 \lambda_{tot}\left(R(x,z,t)\right)\,dz$ for almost all $x\in(0,1)$ and $t\in(0,T)$. Then, we have 
\begin{align*}
 \left\Vert U[S^{\Delta t}_M]- \right.&\left.U[S]\right\Vert_{L^2(\Omega_T)}\\=~&  \left\Vert \frac{\lambda_{tot}(S^{\Delta t}_M)}{A[S^{\Delta t}_M]}- \frac{\lambda_{tot}(S)}{A[S]}\right\Vert_{L^2(\Omega_T)}\\ = ~&   \left\Vert \frac{\lambda_{tot}(S^{\Delta t}_M) \left(A[S]-A[S^{\Delta t}_M]\right)}{A[S^{\Delta t}_M]A[S]}+ \frac{\left(\lambda_{tot}(S^{\Delta t}_M)-\lambda_{tot}(S)\right)A[S^{\Delta t}_M]}{A[S^{\Delta t}_M]A[S]}\right\Vert_{L^2(\Omega_T)}\\ \leq ~&  \left\Vert \frac{\lambda_{tot}(S^{\Delta t}_M) \left(A[S]-A[S^{\Delta t}_M]\right)}{A[S^{\Delta t}_M]A[S]}\right\Vert_{L^2(\Omega_T)}+\left\Vert \frac{\left(\lambda_{tot}(S^{\Delta t}_M)-\lambda_{tot}(S)\right)A[S^{\Delta t}_M]}{A[S^{\Delta t}_M]A[S]}\right\Vert_{L^2(\Omega_T)}.
\end{align*}
Note that $A[S^{\Delta t}_M]A[S]>a^2>0$ using Assumption \ref{ass:Bexistence}(4). Thus, we have
\begin{align*}
 \left\Vert U[S^{\Delta t}_M]\right.&\left.- U[S]\right\Vert_{L^2(\Omega_T)} \leq &\frac{M}{a^2}\left( \left\Vert A[S]-A[S^{\Delta t}_M]\right\Vert_{L^2(\Omega_T)}+ \left\Vert \lambda_{tot}(S^{\Delta t}_M)-\lambda_{tot}(S)\right\Vert_{L^2(\Omega_T)}\right).
\end{align*}
Then, the strong convergence of $\lambda_{tot}$ in \eqref{eq:B-strong-nonlinear} and of $A[S^{\Delta t}_M]$ in \eqref{eq:B-strong-nonlinear1} yield
\begin{align}
 \left\Vert U[S^{\Delta t}_M]- U[S]\right\Vert_{L^2(\Omega_T)} \rightarrow 0.
 \label{eq:B-strong-nonlinear2}
\end{align}
\indent The growth condition on the velocity component $W$ in Lemma \ref{lem:B-more-assumptions}(3) and Lemma \ref{lem:Bapriori1} imply the boundedness of $W[S^{\Delta t}_M]$ in $L^2(\Omega_T)$. Hence, up to a subsequence, there exists a function $k\in  L^2(\Omega_T) $ such that
\begin{align}
\int_0^T \int_{\Omega} W[S^{\Delta t}_M] \partial_z\phi \,dx\,dz\,dt\rightarrow \int_0^T \int_{\Omega} k \, \partial_z\phi \,dx\,dz\,dt
\label{eq:B-Wweak1}
\end{align}
for all test functions $\phi\in L^2(0,T;V_m(\Omega))$ as $m,\,N\rightarrow \infty$. Since $\bigcup_{m\in\mathbb{N}}V_m(\Omega)$ is dense in $H_0^1(\Omega)$, \eqref{eq:B-Wweak1} holds for all test functions $\phi\in L^2(0,T;H_0^1(\Omega))$. To identify the function $k$ in \eqref{eq:B-Wweak1}, we take $\phi\in L^2(0,T;C_0^{2}(\Omega))$ with $\phi(x,\cdot,\cdot)\coloneqq 0 $ for $x\in (-\Delta x,0]\cup[1, 1+\Delta x)$, where $\Delta x>0$ is spatial step size in the $x$-direction. Note that the spatial derivatives in the discrete equation \eqref{eq:Bapproximate1*} correspond to centered differences. Thus, applying summation by parts to the left side of \eqref{eq:B-Wweak1} yields  
\begin{align*}
\int_0^T \int_{\Omega} W[S^{\Delta t}_M] \partial_z\phi \,dx\,dz\,dt=&  -\int_0^T\int_{\Omega} \left(\partial_x \int_0^z U[S^{\Delta t}_M(t)(x,r)]\,dr\right) \partial_z\phi \,dx\,dz\,dt,\\ =& \int_0^T\int_{\Omega} \left(\int_0^z U[S^{\Delta t}_M(t)(x,r)]\,dr\right) \partial^2_{zx}\phi \,dx\,dz\,dt\\&- \dfrac{1}{\Delta x}\int_0^T\int_0^1 \int_{1}^{1+\Delta x} \left(\int_0^z U[S^{\Delta t}_M(t)(x,r)]\,dr\right) \partial_{z}\phi\,dx \,dz\,dt\\ &+\dfrac{1}{\Delta x}\int_0^T\int_0^1 \int_{-\Delta x}^{0} \left(\int_0^z U[S^{\Delta t}_M(t)(x,r)]\,dr\right) \partial_{z}\phi\,dx \,dz\,dt.
\end{align*}
The second and the third term on the right side of the equation above vanish by the choice of the test function $\phi$. We show in the following that the first term on the right side converges to $ \int_0^T\int_{\Omega} \left(\int_0^z U[S(t)(x,r)]\,dr\right) \partial^2_{zx}\phi \,dx\,dz\,dt$. For this, we use H\"older's inequality, Fubini's theorem, and the strong convergence of $U[S^{\Delta t}_M]$ in \eqref{eq:B-strong-nonlinear2} as follows.
\begin{align*} 
 \int_0^T \int_{\Omega}& \left(\int_0^z U[S^{\Delta t}_M](t)(x,r)-U[S](x,r,t)\,dr \right)\partial^2_{zx}\phi \,dx\,dz\,dt\\&\leq \Vert \partial^2_{zx}\phi \Vert_{L^{\infty}(\Omega_T)}\int_0^T \int_{\Omega} \int_0^z \left|U[S^{\Delta t}_M](t)(x,r)-U[S](x,r,t)\right|\,dr \,dx\,dz\,dt\\&\leq \Vert \partial^2_{zx}\phi \Vert_{L^{\infty}(\Omega_T)}\int_0^1\int_0^T \int_0^z\int_0^1  \left|U[S^{\Delta t}_M](t)(x,r)-U[S](x,r,t)\right| \,dx\,dr\,dt\,dz\\&\leq \Vert \partial^2_{zx}\phi \Vert_{L^{\infty}(\Omega_T)}\int_0^T \int_{\Omega}  \left|U[S^{\Delta t}_M](t)(x,r)-U[S](x,r,t)\right| \,dx\,dr\,dt\\& \rightarrow 0.
 \end{align*}
Thus, we have
\begin{align}
 \int_0^T \int_{\Omega} W[S^{\Delta t}_M(t)(x,z)] \partial_z&\phi \,dx\,dz\,dt \nonumber\\&\rightarrow \int_0^T\int_{\Omega} \left(\int_0^z U[S(t)(x,r)]\,dr\right) \partial^2_{zx}\phi \,dx\,dz\,dt
 \label{eq:B-W1}
\end{align}
for all test functions $\phi\in L^2(0,T;H_0^1(\Omega))$ as $m\rightarrow \infty$ and $\Delta t\rightarrow 0$. Combining the results \eqref{eq:B-Wweak1} and \eqref{eq:B-W1} yields 
\begin{align}
 \int_0^T \int_{\Omega} k(x,z,t)\, \partial_z\phi \,dx\,dz\,dt= \int_0^T\int_{\Omega} \left(\int_0^z U[S(t)(x,r)]\,dr\right) \partial^2_{zx}\phi \,dx\,dz.
\end{align}
Thus, we have
\begin{align}
k(x,z,t)=-\partial_x \int_0^z U[S(x,r,t)]dr=W[S(x,z,t)]\in L^2(\Omega_T),
\label{eq:B-WderU}
\end{align}
for almost all $(x,z)\in\Omega$ and $t\in(0,T)$. Substituting \eqref{eq:B-WderU} into \eqref{eq:B-Wweak1} yields the required convergence
 \begin{align}
  \int_0^T \int_{\Omega} W[S^{\Delta t}_M] \partial_z\phi \,dx\,dz \rightarrow \int_0^T\int_{\Omega} W[S] \partial_{z}\phi \,dx\,dz\,dt.
  \label{eq:B-weak-W}
 \end{align}
Now, we prove the strong convergence of the product $U[S^{\Delta t}_M]f(S^{\Delta t}_M)$, i.e., 
\begin{align}
\Vert U[S^{\Delta t}_M]&f(S^{\Delta t}_M)-  U[S]f(S)\Vert_{L^2(\Omega_T)}\nonumber\\=\,& ~\left\Vert U[S]\left(f(S^{\Delta t}_M)-f(S)\right) + f(S^{\Delta t}_M)\left(U[S^{\Delta t}_M]- U[S]\right)\right\Vert_{L^2(\Omega_T)},\nonumber\\ \leq&~ \left\Vert U[S]\left(f(S^{\Delta t}_M)-f(S)\right)\right\Vert_{L^2(\Omega_T)} + \left\Vert f(S^{\Delta t}_M)\left(U[S^{\Delta t}_M]- U[S]\right)\right\Vert_{L^2(\Omega_T)}.
\label{eq:Brink-conv1}
\end{align}
The boundedness of $U$ in the space $L^{\infty}(\Omega_T)$ by Lemma \ref{lem:B-more-assumptions}(1) and the strong convergence of $f$ in \eqref{eq:B-strong-nonlinear} imply
\begin{align}
\left\Vert  U[S]\left(f(S^{\Delta t}_M)-f(S)\right)\right\Vert_{L^2(\Omega_T)}\leq ~ \frac{M}{a}\Vert  f(S^{\Delta t}_M)-f(S)\Vert_{L^2(\Omega_T)} \rightarrow 0.
\label{eq:Brink-conv2}
\end{align}
The boundedness of $f$ in $L^{\infty}(\Omega_T)$ by Assumption \ref{ass:Bexistence}(3) and the strong convergence of $U$ in \eqref{eq:B-strong-nonlinear1} lead to
\begin{align}
 \left\Vert f(S^{\Delta t}_M)\left(U[S^{\Delta t}_M]- U[S]\right)\right\Vert_{L^2(\Omega_T)}\leq ~ M \Vert U[S^{\Delta t}_M]- U[S]\Vert_{L^2(\Omega_T)}\rightarrow 0.
 \label{eq:Brink-conv3}
\end{align}
Substituting \eqref{eq:Brink-conv2} and \eqref{eq:Brink-conv3} into \eqref{eq:Brink-conv1} yields
\begin{align*}
 U[S^{\Delta t}_M]f(S^{\Delta t}_M)\rightarrow  U[S]f(S)\quad \text{ in } L^2(\Omega_T).
\end{align*}

\indent We also prove the weak convergence of the product $W[S^{\Delta t}_M]f(S^{\Delta t}_M)\in L^2(\Omega_T)$. The boundedness of the fractional flow function $f\in L^{\infty}(\Omega_T)$, the growth condition on $W$ in Lemma \ref{lem:B-more-assumptions}(3), and Lemma \ref{lem:Bapriori1} imply the existence of a constant $C > 0$ such that 
\begin{align*}
 \Vert W[S^{\Delta t}_M]f(S^{\Delta t}_M)\Vert_{L^2(\Omega_T)} \leq \frac{2M^2L}{a^2}\Vert\partial_x S^{\Delta t}_M\Vert_{L^2(\Omega_T)} \leq C.
\end{align*}
Hence, there exists a function $q\in L^2(\Omega_T)$ such that, up to a subsequence,
\begin{align}
 \int_0^T \int_{\Omega}  W[S^{\Delta t}_M]f(S^{\Delta t}_M)\, \phi \,dx\,dz\,dt \rightarrow  \int_0^T \int_{\Omega}  q\, \phi \,dx\,dz\,dt,
 \label{eq:B-Wfweak}
\end{align}
for all test functions $\phi\in L^2(0,T;V_m(\Omega))$ as $m,\,N\rightarrow \infty$. Since $\cup_{m\in\mathbb{N}}V_m(\Omega)$ is dense in $H_0^1(\Omega)$, \eqref{eq:B-Wfweak} holds for all test functions $\phi\in L^2(0,T;H_0^1(\Omega))$. To identify $q$, we take a test function $\phi\in L^{\infty}(0,T;C_0^1(\Omega))$ in \eqref{eq:B-Wfweak}. Then we have
\begin{align}
 &\int_0^T \int_{\Omega}  \big(W[S^{\Delta t}_M]f(S^{\Delta t}_M)-  W[S]f(S)\big)\phi\,dx\,dz\,dt \nonumber\\&= \int_0^T \int_{\Omega}  \Big( W[S^{\Delta t}_M]\left(f(S^{\Delta t}_M)-f(S)\right) +  f(S)\left(W[S^{\Delta t}_M]- W[S]\right)\Big)\phi\,dx\,dz\,dt . 
 \label{eq:B-fW1}
\end{align}
The choice of the test function implies $\phi\in L^{\infty}(\Omega_T)$. Thus, the growth condition on $W$ in Lemma \ref{lem:B-more-assumptions}.3, Lemma \ref{lem:Bapriori1}, H\"older's inequality, and the strong convergence of $f$ in \eqref{eq:B-strong-nonlinear} lead to 
\begin{align}
 \int_0^T \int_{\Omega}  W[S^{\Delta t}_M]&\left(f(S^{\Delta t}_M)-f(S)\right)\phi\,dx\,dz\,dt\nonumber \\ & \leq~ \Vert\phi\Vert_{L^{\infty}(\Omega_T)}\Vert W[S^{\Delta t}_M]\Vert_{L^2(\Omega_T)} \Vert f(S^{\Delta t}_M)-f(S)\Vert_{L^2(\Omega_T)}\nonumber\\ & \rightarrow~ 0.
 \label{eq:B-fW2}
\end{align}
The weak convergence of $W$ in \eqref{eq:B-weak-W} implies
\begin{align}
\int_0^T \int_{\Omega} f(S)\phi \left(W[S^{\Delta t}_M]- W[S]\right)\,dx\,dz\,dt\rightarrow 0. 
\label{eq:B-fW3}
\end{align}
Substituting \eqref{eq:B-fW2} and \eqref{eq:B-fW3} into \eqref{eq:B-fW1} yields
\begin{align}
 \int_0^T \int_{\Omega} W[S^{\Delta t}_M]f(S^{\Delta t}_M)\phi\,dx\,dz\,dt\rightarrow\int_0^T \int_{\Omega} W[S]f(S)\phi\,dx\,dz\,dt.
 \label{eq:B-weak-Wf}
\end{align}
By the uniqueness of the limit we obtain $q= W[S]f(S)\in L^2(\Omega_T)$. 

The existence of a function $S\in L^2(0,T;H_0^1(\Omega))$ with $\partial_t S\in L^2(0,T;H_0^1(\Omega))$ and the convergence results \eqref{eq:Bweakconv}, \eqref{eq:B-strong-nonlinear2}, and \eqref{eq:B-weak-Wf} imply that equation \eqref{eq:B-weakdiscreate} converges as $m\rightarrow \infty$ and $\Delta t\rightarrow 0$ to 
\begin{align}
\int_{0}^{T} \int_{\Omega} \partial_t S \phi\,dx\,dz\,dt-& \int_0^{T}\int_{\Omega} U[S]f(S)\partial_x \phi \,dx\,dz\,dt - \int_0^{T}\int_{\Omega} W[S]f(S)\partial_z \phi \,dx\,dz\,dt\nonumber \\ &+\int_{0}^{T} \int_{\Omega}\nabla S\cdot\nabla \phi + \beta \nabla \partial_t S\cdot \nabla \phi\,dx\,dz\,dt=0,
\label{eq:B-weak}
\end{align}
for all test function $\phi\in L^2(0,T;H_0^1(\Omega))$. Hence, the function $S$ satisfies the first condition in Definition \ref{def:Bweaksolution}. \\
\indent Now, we show that the function $S\in H^1(0,T;H_0^1(\Omega))$ satisfies the weak incompressibility equation in Definition \ref{def:Bweaksolution}.
We choose a test function $\phi\in C_0^2(\Omega_T)$, then using \eqref{eq:B-WderU}, we have
\begin{align}
  \int_{0}^{T}\int_{\Omega} W[S]\partial_z\phi \,dx\,dz\,dt=-\int_{0}^{T}\int_{\Omega}\partial_x\int_0^z U[S](x,r,t)\,dr\partial_z\phi\,dx\,dz\,dt.
\end{align}
Applying Gauss' theorem to the right side of the above equation twice yields
\begin{align}
 \int_{0}^{T}\int_{\Omega} W[S]\partial_z\phi \,dx\,dz\,dt=& \int_{0}^{T}\int_{\Omega}\int_0^z U[S](x,r,t)\,dr\,\partial^2_{zx}\phi\,dx\,dz\,dt,\nonumber\\=& - \int_{0}^{T}\int_{\Omega}\partial_z\int_0^z U[S](x,r,t)\,dr\,\partial_{x}\phi\,dx\,dz\,dt.
\label{eq:B-weakincomp1}
\end{align}
For fixed but arbitrary $x\in(0,1)$ and $t\in(0,T)$ we have $\int_0^z U[S](x,r,t)\,dr\in H^2((0,1))$ for any $z\in(0,1)$. Hence, for $\Delta z>0$ applying the Taylor expansion yields
\begin{align*}
 \int_{0}^{T}&\int_{\Omega}\partial_z\int_0^z U[S](x,r,t)\,dr\,\partial_{x}\phi\,dx\,dz\,dt\\ =~ &\frac{1}{\Delta z}  \int_{0}^{T}\int_{\Omega} \Bigl(\int_{0}^{z+\Delta z}U[S](x,r,t)\,dr-\int_{0}^{z}U[S](x,r,t)\,dr\Bigr)\,\partial_{x}\phi\,dx\,dz\,dt+ \mathcal{O}(\Delta z),\\
 =& ~ \frac{1}{\Delta z}  \int_{0}^{T}\int_{\Omega} \int_{z}^{z+\Delta z}U[S](x,r,t)\,dr\partial_{x}\phi\,dx\,dz\,dt + \mathcal{O}(\Delta z).
\end{align*}
Letting $\Delta z\rightarrow 0$, we obtain
\begin{align}
  \int_{0}^{T}\int_{\Omega}\Big(\partial_z\int_0^z U[S](x,r,t)\,dr\Big)\,\partial_{x}\phi\,dx\,dz\,dt=\int_{0}^{T}\int_{\Omega} U[S](x,z,t)\,\partial_{x}\phi\,dx\,dz\,dt.
\label{eq:B-weakincomp2}
\end{align}
Substituting \eqref{eq:B-weakincomp2} into \eqref{eq:B-weakincomp1} yields the required weak incompressibility equation in Definition \ref{def:Bweaksolution}.

Finally, we show that $S(0)=S^0$ almost everywhere. Choosing a test function $\phi\in C^1([0,T],H_0^1(\Omega))$ in \eqref{eq:B-weak} such that $\phi(T)=0$, then applying Gauss' theorem to the first term in equation \eqref{eq:B-weak} yields
\begin{align}
&\int_{0}^{T} \int_{\Omega} S \partial_t \phi\,dx\,dz\,dt- \int_0^{T}\int_{\Omega} U[S]f(S)\partial_x \phi \,dx\,dz\,dt - \int_0^{T}\int_{\Omega} W[S]f(S)\partial_z \phi \,dx\,dz\,dt\nonumber \\ &+\int_{0}^{T} \int_{\Omega}\nabla S\cdot\nabla \phi + \beta \nabla \partial_t S\cdot \nabla \phi\,dx\,dz\,dt=\int_{\Omega} S(0) \phi(0)\,dx\,dz.
\label{eq:B-weak1}
\end{align}
Applying summation by parts to the first term in equation \eqref{eq:B-weakdiscreate} yields
\begin{align}
\dfrac{1}{\Delta t} &\int_{0}^{T} \int_{\Omega} \left(\phi(t)-\phi(t-\Delta t) \right)S^{\Delta t}_M\,dx\,dz\,dt\nonumber\\&- \int_0^{T}\int_{\Omega} U[S^{\Delta t}_M]f(S^{\Delta t}_M)\partial_x \phi \,dx\,dz\,dt - \int_0^{T}\int_{\Omega} W[S^{\Delta t}_M]f(S^{\Delta t}_M)\partial_z \phi \,dx\,dz\,dt\nonumber \\ &+\int_{0}^{T} \int_{\Omega}\nabla S^{\Delta t}_M\cdot\nabla \phi + \dfrac{\beta}{\Delta t} \left(\nabla\phi(t)-\nabla\phi(t-\Delta t) \right)\cdot \nabla S^{\Delta t}_M\,dx\,dz\,dt\nonumber\\&= \int_{\Omega} S_m^0\phi(0)\,dx\,dz.
\label{eq:B-initial}
\end{align}
Letting $m\rightarrow \infty$ and $\Delta t\rightarrow 0$, equation \eqref{eq:B-initial} converges, up to a subsequence, to
\begin{align}
&\int_{0}^{T} \int_{\Omega} S \partial_t \phi\,dx\,dz\,dt- \int_0^{T}\int_{\Omega} U[S]f(S)\partial_x \phi \,dx\,dz\,dt - \int_0^{T}\int_{\Omega} W[S]f(S)\partial_z \phi \,dx\,dz\,dt\nonumber \\ &+\int_{0}^{T} \int_{\Omega}\nabla S\cdot\nabla \phi + \beta \nabla \partial_t S\cdot \nabla \phi\,dx\,dz\,dt=\int_{\Omega} S^0 \phi(0)\,dx\,dz,
\label{eq:B-weak2}
\end{align}
since $S_m^0\rightarrow S^0$ in $L^2(\Omega)$ as $m\rightarrow \infty$. As $\phi(0)$ is arbitrarily chosen, comparing equation \eqref{eq:B-weak1} and \eqref{eq:B-weak2} yields that $S(0)=S^0$ almost everywhere. Hence, the function $S$ satisfies the third condition in Definition \ref{def:Bweaksolution}, which implies that $S$ is a weak solution of the initial boundary value problem \eqref{eq:model}, \eqref{eq:velocity}, \eqref{eq:incompressible} and \eqref{eq:Bibc}.
\end{proof}

\begin{remark}
\begin{enumerate}
 \item Proving uniqueness of weak solutions for the initial boundary value problem \eqref{eq:model}, \eqref{eq:velocity}, \eqref{eq:incompressible} and \eqref{eq:Bibc} requires that weak solutions satisfy $\partial_x S \in L^{\infty}(\Omega_T)$. However, proving this property is still unfeasible as the regularization theory in \cite{Gilbarg,Ladyz} is not applicable. 
 
 \item Uniqueness can be guaranteed for the initial boundary value problem \eqref{eq:model}, \eqref{eq:velocity}, \eqref{eq:incompressible} and \eqref{eq:Bibc} with a linear choice of the fractional flow function $f(S)=S$ and the horizontal velocity component $U(S)=S$ under the assumption that weak solutions satisfy $\partial_x S \in L^{r}(\Omega_T),\,r>2$.
\end{enumerate}
 \end{remark}

\section{Numerical Examples}
\label{sec:numerical}
In this section, we investigate using numerical examples the effect of letting the regularization parameter $\beta \rightarrow 0$ in the BVE-model tends to zero. In addition, we show the difference between solutions of the zero-limit of the BVE-model \eqref{eq:model}, \eqref{eq:velocity} and those of the DVE-model \eqref{eq:YortsosModel}, \eqref{eq:velocity} (the BVE-model with $\beta=0$). 

\begin{remark}
Note that the BVE-model \eqref{eq:model}, \eqref{eq:velocity} reduces to the DVE-model \eqref{eq:YortsosModel}, \eqref{eq:velocity} as the regularization parameter $\beta \rightarrow 0$. However, the estimates in Section \ref{sec:a priori} depend on $\beta$ and blow up as $\beta\rightarrow 0$, in particular the estimates on $\nabla S$. Therefore, saturation in the limit $\beta\rightarrow 0$ is not expected to have enough regularity to be a standard weak solution the DVE-model \eqref{eq:YortsosModel}, \eqref{eq:velocity}.  
\end{remark}
 
For the numerical examples, we consider the BVE-model \eqref{eq:model}, \eqref{eq:velocity} with a nonlinear diffusion function $H=H(S)$ such that
\begin{eqnarray}
\partial_{t}S +\partial_{x}\Bigl(f(S)U[S]\Bigr)+\partial_{z}\Bigl(f(S)W[S]\Bigr)-\beta \nabla\cdot\Big(H(S)\nabla S \Big) -\beta^2 \Delta \partial_{t}S=0
\label{eq:BVE}
\end{eqnarray}
in $\Omega\times(0,T)$. Here, 
\begin{equation}
U[S]=\dfrac{\lambda_{tot}(S)}{\int_{0}^{1}\lambda_{tot}(S) dz},\quad \quad W[S]= -\partial_{x}\int_{0}^{z}U[S(\cdot,r,\cdot)]dr,
\label{eq:velocity1}
\end{equation}
and 
\begin{align}
 f(S)= \frac{MS^2}{MS^2+(1-S)^2}, \quad\quad H(S)= \frac{MS^2(1-S)^2}{MS^2+(1-S)^2},
 \label{eq:diffusion}
\end{align}
where $M$ is the viscosity ratio of the defending phase and the invading phase. We also consider the initial and boundary conditions
\begin{align}
\begin{array}{rll}
 S(\cdot,\cdot,0)&=S_{0}, &\text{ in } \Omega, \\
 S&=S_{\text{inflow}}, &\text{ on } \{0\}\times (0,1)\times [0,T],\\
 W&= 0,  &\text{ on }(0,1)\times\{0,1\}\times [0,T].
\end{array}
\label{eq:IBC}
\end{align}
Note that the second condition in \eqref{eq:IBC} corresponds to a steady flow at the left boundary of the domain, however, the third condition corresponds to impermeable upper and lower boundaries of the domain. In the following examples we choose the initial condition
\begin{align*}
 S_0(x,z)=g(x)S_{\text{inflow}}(z),
\end{align*}
where \begin{align}
 g(x)=\dfrac{(1-x)^2}{10^5x^2+(1-x)^2}\quad\text{ and }\quad S_{\text{inflow}}(z)=\left\{ 
\begin{array}{c l l}
	0 \quad &: & z\leq \frac{1}{4} \text{ and } z>\frac{3}{4},\\
	0.9 \quad &: &\frac{1}{4}<z\leq \frac{3}{4}.
\end{array} \right.
\label{eq:inflow}
 \end{align}

We discretize the nonlocal BVE-model \eqref{eq:BVE}, \eqref{eq:velocity1} by applying a mass-conservative finite-volume scheme as described in \cite{Armiti-Juber2018}. The scheme is based on a Cartesian grid with number of vertical cells $N_z$ significantly less than that in the horizontal direction $N_x$ that fits to the case of flat domains. In the following two examples, we use a grid of $2000\times 40$ elements, viscosity ratio $M=2$ and end time $T=0.5$.

\begin{figure}
\centering
\subfigure[$\beta^2=10^{-2}$]{
\includegraphics[scale=0.38]{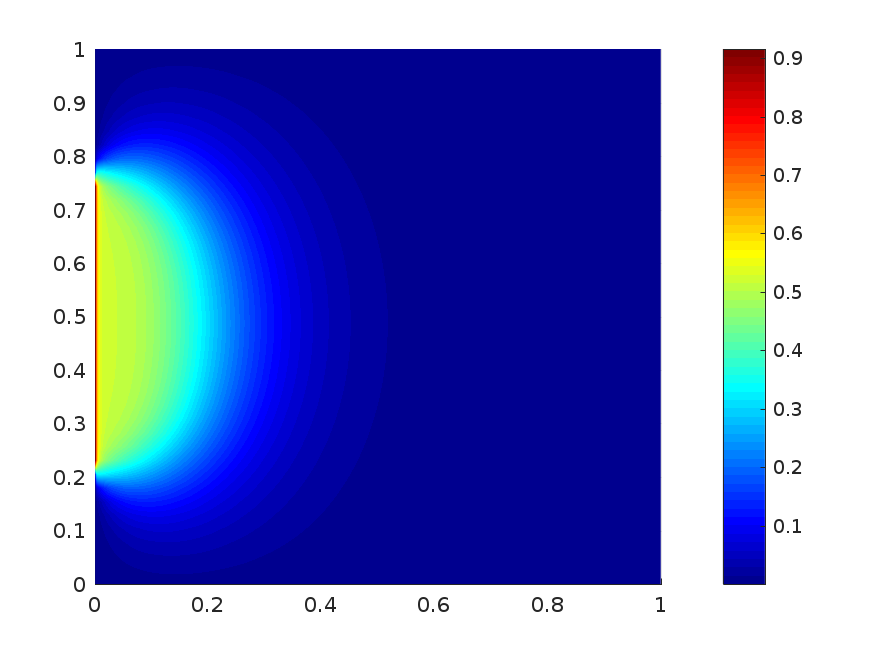}
\label{subfig:beta1}
}\hspace{-0.1cm}
\subfigure[$\beta^2=10^{-3}$]{
\includegraphics[scale=0.38]{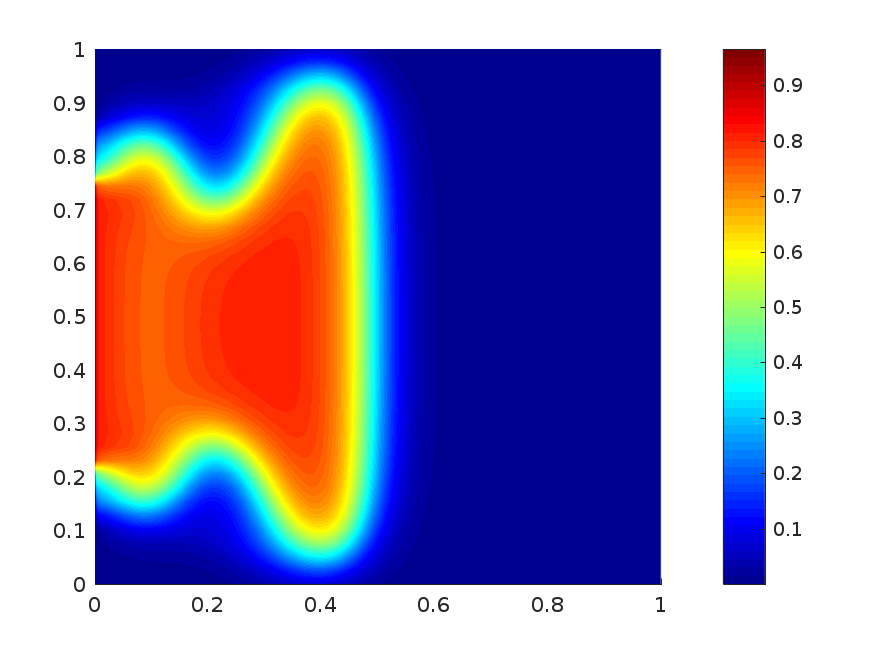}
\label{subfig:beta2}
}\\
\subfigure[$\beta^2=10^{-4}$]{
\includegraphics[scale=0.38]{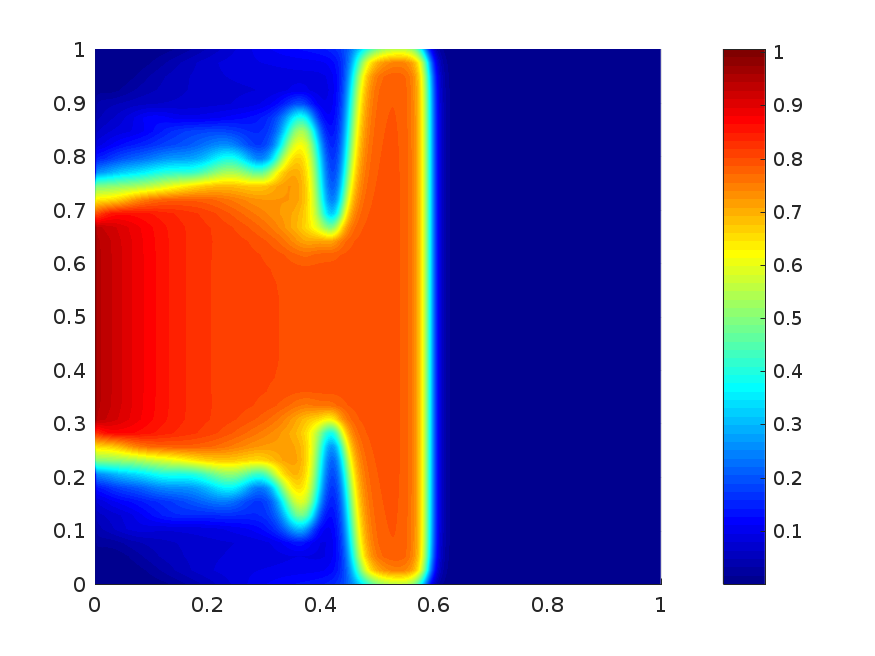}
\label{subfig:beta3}
}\hspace{-0.1cm}
\subfigure[$\beta^2=10^{-5}$]{
\includegraphics[scale=0.38]{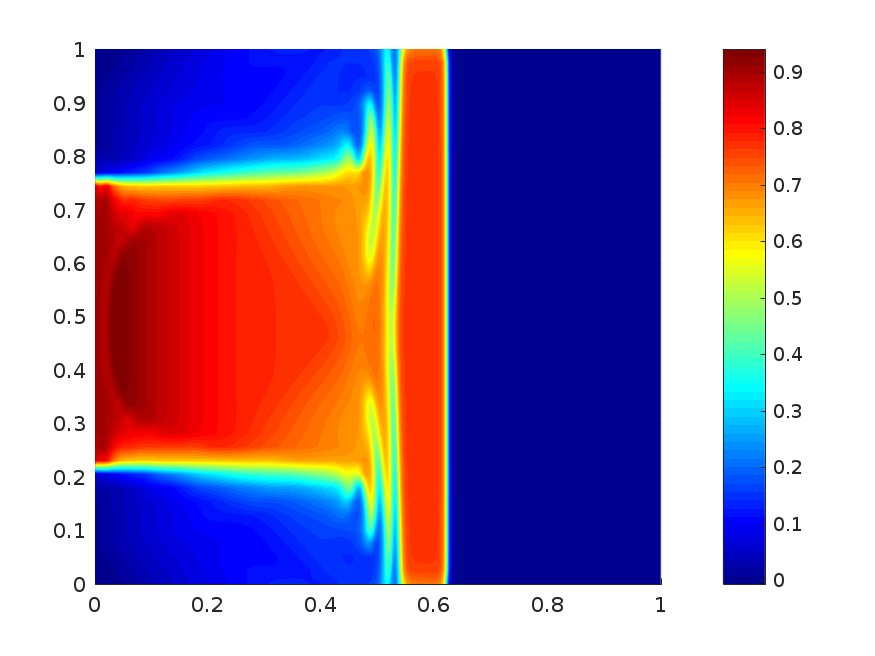}
\label{subfig:beta4}
}
\caption{Numerical solutions of the BVE-model \eqref{eq:BVE}, \eqref{eq:velocity1} for decreasing parameters $\beta^2\in\{10^{-2},\,10^{-3},\,10^{-4},\,10^{-5}\}$ using a $2000\times40$ grid, $M=2$ and $T=0.5$.}
\label{fig:comparison1}
\end{figure}

\textbf{Example 1:} In this example, we show the effect of reducing the regularization parameter $\beta\rightarrow 0$ on the numerical solutions of the BVE-model \eqref{eq:BVE}, \eqref{eq:velocity1}. In Figures \ref{subfig:beta1}-\ref{subfig:beta4}, we present the numerical solutions of the BVE-model \eqref{eq:BVE}, \eqref{eq:velocity1} using the parameters $\beta^2\in\{10^{-2},\,10^{-3},\,10^{-4},\,10^{-5}\}$, respectively. The results in Figure \ref{fig:comparison1} show a high diffusional effect on saturation solution for $\beta^2=10^{-2}$ that decreases with $\beta$. In fact, it is noticeable that the saturation consists of a sharp moving front as $\beta\rightarrow 0$. This result matches with the a priori estimates on $\nabla S$ in Section \ref{sec:a priori}, which blow up as $\beta\rightarrow 0$.

\textbf{Example 2:} As the quasi-parabolic BVE-model \eqref{eq:BVE}, \eqref{eq:velocity1} reduces to the nonlocal transport equation \eqref{eq:YortsosModel}, \eqref{eq:velocity} proposed in \cite{Yortsos} when $\beta \rightarrow 0$, we show in this example that numerical solution of the BVE-model with $\beta\rightarrow 0$ differs from that of the DVE-model. In Figure \ref{subfig:BVE} we present the numerical solution of the BVE-model with $\beta^2=10^{-6}$, while in Figure \ref{subfig:DVE} we show the numerical solution of the DVE-model.

In contrast to the DVE-model, Figure \ref{fig:comparison2} shows that the BVE-model describes saturation overshoots. In addition, the spreading speed of the inflowing fluid using the BVE-model is smaller than that using the DVE-model. This is a consequence of the saturation overshoots phenomenon, which is mathematically identified by undercompressive waves that are known to be slower than classical compressive waves. This result was expected by Yortsos and Salin in \cite{YortsosSalin}, where they developed different selection principles on finding upper bounds on the speed of the mixing zone in the case for miscible displacement.  

\begin{figure}
\centering
\subfigure[$\beta^2=10^{-6}$]{
\includegraphics[scale=0.38]{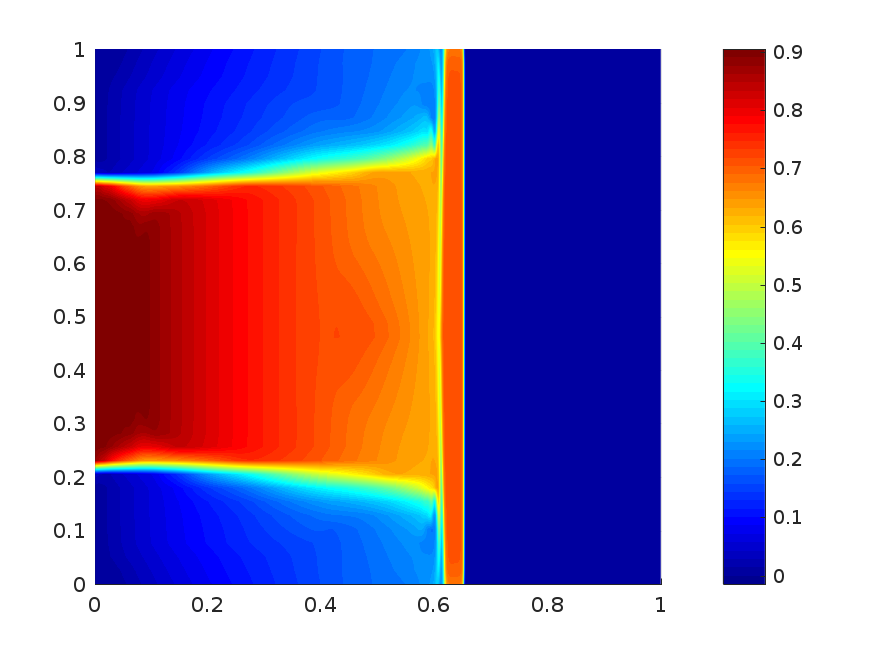}
\label{subfig:BVE}
}\hspace{-0.1cm}
\subfigure[$\beta^2=0$]{
\includegraphics[scale=0.38]{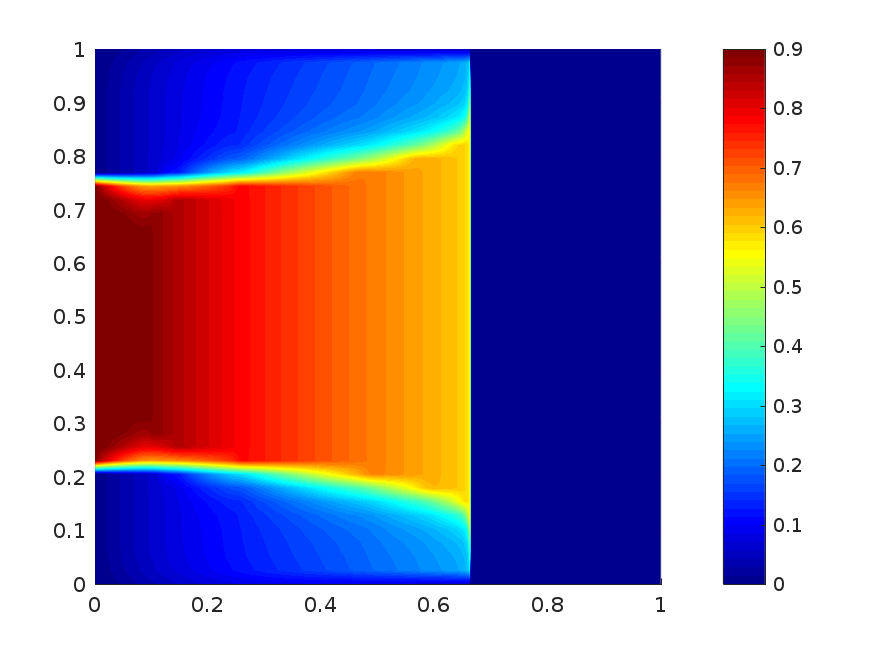}
\label{subfig:DVE}
}
\caption{Numerical solutions of the BVE-model \eqref{eq:BVE} for $\beta^2=10^{-6}$ in (a) and for $\beta=0$ in (b), using a $2000\times40$ grid, $M=2$ and $T=0.5$.}
\label{fig:comparison2}
\end{figure}

%
%
%
%
%

\bibliographystyle{plain}
\bibliography{Existence-Brinkman}

\end{document}